\documentclass[12pt]{amsart}
\usepackage{amscd,amsmath,amsthm,amssymb}
\usepackage{color}
\usepackage{amsfonts,amssymb,amscd,amsmath,enumerate,verbatim,enumitem}
\usepackage{lineno}
 \def\NZQ{\mathbb}               
 \def\NN{{\NZQ N}}
 
 \def\ZZ{{\NZQ Z}}
 \def\RR{{\NZQ R}}

 \def\KK{{\NZQ K}}
 
 \def\frk{\mathfrak}               

 \def\mm{{\frk m}}

 \def\G{{\mathcal G}}

 \def\xb{{\mathbf x}}

 \def\eb{{\mathbf e}}
 \def\ub{{\mathbf u}}
 \def\vb{{\mathbf v}}
 \def\opn#1#2{\def#1{\operatorname{#2}}} 
 \opn\chara{char} \opn\length{\ell} \opn\pd{pd} \opn\rk{rk}
 \opn\projdim{proj\,dim} \opn\injdim{inj\,dim} \opn\rank{rank}
 \opn\depth{depth} \opn\grade{grade} \opn\height{height}
 \opn\embdim{emb\,dim} \opn\codim{codim}
 
 \opn\Tr{Tr} \opn\bigrank{big\,rank}
 \opn\superheight{superheight}\opn\lcm{lcm}
 \opn\trdeg{tr\,deg}
 \opn\reg{reg} \opn\lreg{lreg} \opn\ini{in} \opn\lpd{lpd}
 \opn\size{size} \opn\sdepth{sdepth}
 \opn\link{link}\opn\fdepth{fdepth}\opn\lex{lex}
 \opn\tr{tr}
 \opn\div{div} \opn\Div{Div} \opn\cl{cl} \opn\Cl{Cl}
 \opn\Spec{Spec} \opn\Supp{Supp} \opn\supp{supp} \opn\Sing{Sing}
 \opn\Ass{Ass} \opn\Min{Min}\opn\Mon{Mon}
 \opn\Ann{Ann} \opn\Rad{Rad} \opn\Soc{Soc}
 \opn\Im{Im} \opn\Ker{Ker} \opn\Coker{Coker} \opn\Am{Am}
 \opn\Hom{Hom} \opn\Tor{Tor} \opn\Ext{Ext} \opn\End{End}
 \opn\Aut{Aut} \opn\id{id}
 
 \opn\nat{nat}
 \opn\pff{pf}
 \opn\Pf{Pf} \opn\GL{GL} \opn\SL{SL} \opn\mod{mod} \opn\ord{ord}
 \opn\Gin{Gin} \opn\Hilb{Hilb}\opn\sort{sort}
 \opn\PF{PF}\opn\Ap{Ap}
 \opn\aff{aff} \opn
 \con{conv} \opn\relint{relint} \opn\st{st}
 \opn\lk{lk} \opn\cn{cn} \opn\core{core} \opn\vol{vol}  \opn\inp{inp} \opn\nilpot{nilpot}
 \opn\link{link} \opn\star{star}\opn\lex{lex}\opn\set{set}
 \opn\width{wd}
 \opn\Fr{F}
 \opn\QF{QF}
 \opn\G{G}
 \opn\type{type}\opn\res{res}
 \opn\gr{gr}
 
 \def\pot#1#2{#1[\kern-0.28ex[#2]\kern-0.28ex]}

 \opn\dirlim{\underrightarrow{\lim}}
 \opn\inivlim{\underleftarrow{\lim}}
 \def\Implies{\ifmmode\Longrightarrow \else
         \unskip${}\Longrightarrow{}$\ignorespaces\fi}
 \def\implies{\ifmmode\Rightarrow \else
         \unskip${}\Rightarrow{}$\ignorespaces\fi}
 \def\iff{\ifmmode\Longleftrightarrow \else
         \unskip${}\Longleftrightarrow{}$\ignorespaces\fi}

 \let\:=\colon
 \newtheorem{Theorem}{Theorem}[section]
 \newtheorem{Lemma}[Theorem]{Lemma}
 \newtheorem{Corollary}[Theorem]{Corollary}
 \newtheorem{Proposition}[Theorem]{Proposition}

 \newtheorem{Example}[Theorem]{Example}

\newtheorem{Claim}[Theorem]{Claim}
 \let\epsilon\varepsilon
 \let\kappa=\varkappa
 \def\qed{\ifhmode\textqed\fi
       \ifmmode\ifinner\quad\qedsymbol\else\dispqed\fi\fi}
 \def\textqed{\unskip\nobreak\penalty50
        \hskip2em\hbox{}\nobreak\hfil\qedsymbol
        \parfillskip=0pt \finalhyphendemerits=0}
 \def\dispqed{\rlap{\qquad\qedsymbol}}
 
 \opn\dis{dis}
 \def\pnt{{\raise0.5mm\hbox{\large\bf.}}}
 
 \opn\Lex{Lex}

\begin{document}
\title {Nearly Gorenstein rings arising from finite graphs}

\author {Takayuki Hibi and Dumitru I.\ Stamate}

\address{Takayuki Hibi, Department of Pure and Applied Mathematics, Graduate School of Information Science and Technology,
Osaka University, Suita, Osaka 565-0871, Japan}
\email{hibi@math.sci.osaka-u.ac.jp}

\address{Dumitru I. Stamate, Faculty of Mathematics and Computer Science, University of Bucharest, Str. Academiei 14, Bucharest -- 010014, Romania }
\email{dumitru.stamate@fmi.unibuc.ro}

\dedicatory{ }

\begin{abstract}
The classification of complete multipartite graphs whose edge rings are nearly Gorenstein as well as  that of finite perfect graphs whose stable set rings are nearly Gorenstein is achieved.
\end{abstract}

\thanks{}

 


\maketitle


Gorenstein graded algebras associated to combinatorial objects like graphs or simplicial complexes have attracted a lot of interest.  See, e.g., \cite{DeNegri-Hibi}, \cite{OHjct}, \cite{BH}.
Recently several extensions of the class of Gorenstein rings (inside the class of Cohen--Macaulay rings) have been discussed in, e.g., \cite{GTT}, \cite{HHS}, hence it is natural to search for the combinatorial counterpart.

According to \cite{HHS}, when $R$ is a Cohen--Macaulay graded $\KK$-algebra over the field $\KK$ with canonical module $\omega_R$,  it is called {\em nearly Gorenstein} if the canonical trace ideal  $\tr(\omega_R)$ contains the maximal graded ideal $\mm_R$ of $R$.
Here  $\tr(\omega_R)$ is the ideal generated by the image of $\omega_R$  through all homomorphism of $R$-modules into $R$. As $\tr(\omega_R)$ describes the non-Gorenstein locus of $R$  (\cite[Lemma 2.1]{HHS}), one has $\tr(\omega_R)=R$ if and only if $R$ is a Gorenstein ring.

In the present paper we initiate the study of nearly Gorenstein rings belonging to two classes of algebras associated to graphs.
Throughout, $\KK$ is any field. 
Assume $G$ is a simple graph (it possesses no loops or multiple edges) with vertex set $V(G)=[d]:=\{1,\dots, d\}$. 

The {\em edge ring} $\KK[G]$ is the $\KK$-subalgebra of the polynomial ring $\KK[x_1,\dots, x_d]$ generated by the monomials $x_i x_j$ for all edges $\{i,j\} \in E(G)$.
When $V(G)$ can be partitioned $V(G)=\sqcup_{k=1}^n V_k$ with $n\geq 2$ and  $|V_k|=r_k$ for $k=1,\dots, n$ such that $E(G)$ consists of all the pairs $\{i, j\}$ with $i\in V_a$ and $j\in V_b$ for $1\leq a <b\leq n$, we say that $G$ is a {\em complete multipartite graph} of type $r_1,\dots, r_n$  which is denoted  $K_{r_1,\dots, r_n}$.
Related algebraic properties for these graphs have been recently studied in \cite{HM-2020} and \cite{HM-2021}.
In Proposition~\ref{prop:ng-bipartite} and in Theorem~\ref{thm:ng-graphs} we prove the following result.

\medskip

\noindent{\bf Theorem A.} Assume $G=K_{r_1,\dots, r_n}$. Set $R=\KK [G]$. Then
 \begin{enumerate}
\item if $n=2$ and $1\leq r_1\leq r_2$, the ring $R$ is nearly Gorenstein if and only if  $r_1=1$, or $r_2\in \{r_1, r_1+1\}$. 
\item if $n\geq 3$ the ring $R$ is nearly Gorenstein if and only if $R$ is Gorenstein.
\end{enumerate}
Since   Ohsugi and Hibi in \cite{OH-Ill} have  explicitly listed the complete multipartite graphs whose edge ring is Gorenstein (see Theorem~\ref{thm:gore-multipartite} below), Theorem A offers a full description for the nearly Gorenstein property, as well.

\medskip 

The other class of algebras we consider deals with the stable sets in $G$. A nonempty set $W$ of vertices is called {\em stable} (or {\em independent})  if there is no edge $\{i, j\}$ in $G$ with $i, j\in W$. The {\em stable set ring} of $G$ denoted ${\rm Stab}_\KK(G)$ is the $\KK$-subalgebra in the polynomial ring $\KK[x_1,\dots, x_d,t]$ generated by those monomials $(\prod_{i\in W} x_i)\cdot t$ with $W$ any stable set in $G$. When $G$ is a perfect graph, it is known \cite{OH} that ${\rm Stab}_\KK (G)$ is Cohen--Macaulay, and that it is Gorenstein if and only if all maximal cliques of $G$ have the same cardinality \cite{OHjct}. Recall that a set $C\subset V(G)$ is called a clique if the subgraph induced by $C$ is a complete graph.

The size of the maximal cliques in $G$ is also relevant to describe in Theorem \ref{ngstable} for which perfect graphs the  algebra ${\rm Stab}_\KK(G)$ is nearly Gorenstein. We prove the following.

\medskip

\noindent{\bf Theorem B.}
Let $G$ be a perfect graph and  $G_1,\dots, G_s$ its connected components. Let $\delta_i$ denote the maximal cardinality  of cliques of $G_i$. Then ${\rm Stab}_\KK(G)$ is nearly Gorenstein  if and only if for each $G_i$ its maximal cliques have the same cardinality  and $|\delta_i-\delta_j|\leq  1$ for $1\leq i<j\leq s$.

\medskip

To prove Theorems A and B we observe that  the algebras $R$ which occur are Cohen--Macaulay domains,  so $\omega_R$ can be identified with an ideal in $R$.  By \cite[Lemma 1.1]{HHS}, its trace can be computed as 
\begin{eqnarray*}
\tr(\omega_R) &=& \omega_R \cdot \omega_R^{-1}, \text{ where} \\
 \omega_R^{-1} &=& \{x\in Q(R): x\cdot \omega_R \subseteq R\}
\end{eqnarray*} 
is the anti-canonical ideal of $R$ and $Q(R)$ denotes the field of fractions of $R$. 

We refer the reader to \cite{BB} and \cite{BH} for the undefined  graph or algebraic notions.

\section{Edge rings}

In this section unless stated otherwise $G=K_{r_1,\dots r_n}$ is the complete multipartite graph on $[d]$  with vertices partitioned $V(G)=V_1\sqcup\dots \sqcup V_n$, $n\geq 2$, $|V_k|=r_k$ for all $k$.    In this context $d=\sum_{k=1}^n r_k$ and without loss of generality, we will always assume  that $1\leq r_1\leq \dots \leq r_n$.

The graph $G$ satisfies the so called odd cycle condition, i.e. for any two odd cycles in $G$ which have no common vertex there is a bridge between them. Indeed, when $n=2$ there is no odd cycle and anything to prove. Assume $n\geq 3$, and $C_1$ and $C_2$ be two disjoint  odd cycles in $G$. Since $G$ is multipartite, each of these contains vertices from at least two of the components $V_1,\dots, V_n$, so one finds $v\in C_1\cap V_a$ and $w\in C_2 \cap V_b$ with $a\neq b$. Then $vw$ is an edge in $G$ and a bridge between $C_1$ and $C_2$.
Consequently, by \cite{OH-jalg} the edge ring 
$$
R=\KK[G]=\KK[x_i x_j: i\in V_a, j\in V_b, 1\leq a<b\leq n]\subset \KK[x_1,\dots, x_d]
$$ is normal, hence a Cohen--Macaulay domain (\cite{Hochster}).
Before we address the nearly Gorenstein property, we recall that Ohsugi and Hibi \cite{OH-Ill} classified the complete multipartite edge rings which are Gorenstein. With notation as above, their result is the following.

\begin{Theorem} (Ohsugi, Hibi \cite[Remark 2.8]{OH-Ill})
\label{thm:gore-multipartite}
The edge ring of the complete multipartite graph $K_{r_1,\dots, r_n}$ is Gorenstein if and only if 
\begin{enumerate}
\item $n=2$ and $(r_1,r_2) \in \{ (1,m), (m,m): m\geq 1 \}$, or
\item $n=3$ and $1\leq r_1\leq r_2 \leq r_2 \leq 2$, or
\item $n=4$ and $r_1=r_2=r_3=r_4=1$.
\end{enumerate}
\end{Theorem}
For some complete multipartite graphs the edge ring fits into classes of algebras for which the nearly Gorenstein property is already understood. 

\begin{Example}
\label{edge-rings-as-sqfreeveronese}
{\em
When $r_1=\dots=r_n=1$, the edge ring $R$ is the squarefree Veronese subalgebra of degree $2$ in the polynomial ring $\KK[x_1,\dots, x_n]$, and according to \cite[Theorem 4.14]{HHS}, $R$ is nearly Gorenstein if and only if it is Gorenstein. The latter property holds if and only if $n\leq 4$, by using work of  De Negri and Hibi \cite{DeNegri-Hibi}, or Bruns, Vasconcelos and Villarreal \cite{BVV}.}
\end{Example}

\begin{Example}
\label{edge-rings-as-hibi-rings}{\em 
According to Higashitani and Matsushita \cite[Proposition 2.2]{HM-2020}, when $n=2$, or when $n=3$ and $r_1=1$, the corresponding edge ring is isomorphic to a Hibi ring, and  for the latter the nearly Gorenstein property is described in \cite{HHS}. We refer to \cite{TH87} for background on Hibi rings.}
\end{Example}

\begin{Theorem}(\cite[Theorem 5.4]{HHS}, \cite{TH87})
\label{thm:hibirings-ng}
Let $P$ be a finite poset. Then the Hibi ring $R$ of the distributive lattice of the order ideals in $P$ is nearly Gorenstein if and only if $P$ is the disjoint union of pure connected posets $P_1,\dots, P_q$ such that $|\rank(P_i)-\rank(P_j)| \leq 1$ for $1\leq i<j\leq q$.

In particular, $R$ is a Gorenstein ring if and only if $P$ is pure.
\end{Theorem}

Based on that, when $G$ is a complete bipartite graph we obtain the following classification.

\begin{Proposition}\label{prop:ng-bipartite}
Let $G=K_{r_1, r_2}$ be the complete bipartite graph with $1\leq r_1\leq r_2$.
Then the edge ring  $\KK[G]$ is nearly Gorenstein if and only if $r_1=1$, or $r_1 \geq 2$ and $r_2\in \{r_1, r_1+1\}$.

When $2\leq r_1= r_2-1$, the ring $\KK[G]$ is nearly Gorenstein and not Gorenstein.
\end{Proposition}

\begin{proof}
By \cite[Proposition 2.2]{HM-2020}, $\KK[G]$ is isomorphic to the Hibi ring associated to the distributive  lattice of order ideals in the poset $P$ which consists of  two disjoint chains with $r_1-1$ and $r_2-1$ elements, respectively. By Theorem~\ref{thm:hibirings-ng}, $\KK[G]$ is nearly Gorenstein if and only if $r_1=1$, or $r_1\geq 2$ and $r_2\in \{r_1, r_1+1\}$.
\end{proof}

For non-bipartite graphs we prove the following result.

\begin{Theorem}
\label{thm:ng-graphs}
Let $R$ be the edge ring of a complete multipartite graph $K_{r_1,\dots, r_n}$ with $n\geq 3$. The following statements are equivalent:
\begin{enumerate}
\item[(i)] $R$ is a Gorenstein ring;
\item[(ii)] $R$ is a nearly Gorenstein ring.
\end{enumerate}
\end{Theorem}

\begin{proof}
Clearly, $(i)\implies (ii)$. 
We'll  prove the converse.

When $n=3$ and $r_1=1\leq r_2\leq r_3$,  by \cite[Proposition 2.2]{HM-2020} the ring $R$  is isomorphic to the Hibi ring associated to the distributive  lattice of order ideals in a poset $Q$  with maximal chains $q_1<\dots <q_{r_1}$, $q_{r_1+1}<\dots<q_{r_1+r_2}$ and $q_1<q_{r_1+r_2}$.  The poset $Q$ is connected, hence $R$ is nearly Gorenstein if and only if it is Gorenstein, i.e. $1=r_1\leq r_2\leq r_3\leq 2$.
  
We now consider the remaining  cases: either $n=3$ and $r_1\geq 2$, or $n\geq 4$. 
 Assume, by contradiction that  $R$ is nearly Gorenstein and not Gorenstein, i.e. 
\begin{equation}
\label{eq:tracemax}
\tr(\omega_R)=\mm_R.
\end{equation}
 
The monomials in $R$ and $\omega_R$ have a nice combinatorial description as feasable integer solutions to some systems of inequalities. This can be described as follows.
We denote $H=\sum_{\{i,j\}\in E(G)} \NN (\eb_i+\eb_j) \subset \NN^d$ the affine semigroup generated by the  columns of the vertex-edge incidence  matrix for $G$, and $\mathcal{C}=\RR_+H$ the rational cone over $H$.

For $\ub=(u_1,\dots, u_d)\in \NN^d$, it follows from \cite{OH-jalg} and \cite[Proposition 3.4]{Villarreal} 
 that $\ub \in H$ (equivalently, $\xb^\ub \in R$) if and only if 
\begin{eqnarray}
\label{eq:r}
 \sum_{i=1}^d u_i  \equiv 0 \mod 2, \\
\nonumber u_1,\dots, u_d \geq 0,  \quad \text{and}\\
\nonumber \sum_{i\notin V_k} u_i \geq \sum_{j\in V_k} u_j \text{ for all } k=1,\dots, n.
\end{eqnarray}
The latter inequalities are equivalent to 
\begin{equation}
\label{two-r}\sum_{i=1}^d u_i \geq 2\sum_{j\in V_k} u_j, \text{ for } k=1,\dots, n.
\end{equation}

Since $R$ is normal, by \cite{Danilov}, \cite{Stanley} (see also \cite[Theorem 6.3.5(b)]{BH}), a $\KK$-basis for $\omega_R$ is given by the monomials $\xb^\ub$ where $\ub=(u_1,\dots, u_d) \in \ZZ^d$ satisfies
\begin{eqnarray}
\label{eq:omega}
\label{zero}\sum_{i=1}^d u_i  \equiv 0 \mod 2, \\
\label{one}  u_1,\dots, u_d \geq 1,  \text{and}\\
\label{two}\sum_{i=1}^d  u_i \geq 2+2\sum_{j\in V_k} u_j, \text{ for } k=1,\dots, n.
\end{eqnarray}

From the equations above it is easy to see that if the monomial $\xb^\ub$ is in $R$ or in $\omega_R$, we can permute the exponents $x_i$ and $x_j$ whenever $i, j \in V_k$ for some $k$, and we obtain another monomial in $R$, or in $\omega_R$, respectively.

In what follows $\ub=(u_1,\dots, u_d)$ and $\vb=(v_1,\dots, v_d)$.

For a monomial $\xb^\ub \in \omega_R$ and $1\leq k\leq n$ we say that $V_k$ (or simply, $k$)  is a \textit{heavy component} in $\ub$ if
 \begin{equation}
\label{eq:heavy}\sum_{i=1}^d u_i=2+ 2\sum_{j\in V_k}u_j.
\end{equation} 

\begin{Claim}  {\em For any $\xb^\ub \in \omega_R$ there exist at most two heavy components in $\ub$. In particular, there is at least one non-heavy component in $\ub$. }\end{Claim}

\begin{proof}
Indeed, if $k_1<k_2<k_3$ are heavy components in $\ub$, then by adding the equations \eqref{eq:heavy} for these indices we get
$$
3\sum_{i=1}^d u_i= 6+ \sum_{j\in V_{k_1}\cup V_{k_2}\cup V_{k_3}}2u_j,
$$
If $n=3$, then $\sum_{i=1}^d u_i=6$. Since $u_i\geq r_i\geq 2$ for all $i$, we infer that $r_1=r_2=r_3=2$, and $\KK[G]$ is a Gorenstein ring (by Theorem \ref{thm:gore-multipartite}), which is not the case. 

If $n\geq 4$, then $\sum_{i=1}^d  u_i< 6$. As $\sum_{i=1}^d u_i$ is even, we get that $n=4$ and $r_1=r_2=r_3=r_4=1$.  Example \ref{edge-rings-as-sqfreeveronese} implies that $R$ is a Gorenstein ring, which is false.
\end{proof}

\begin{Claim} {\em For any $1\leq i\leq d$ there exists a monomial $\xb^\ub \in \omega_R$ such that $u_i=1$. }
\end{Claim}

\begin{proof} We fix $i$ and we denote $a_i=\min \{u_i: \prod x_i^{u_i} \in \omega_R\}$. By \eqref{one}, $a_i \geq 1$. Assume $a_i\geq 2$, and say $i\in V_k$.

If $r_k>1$, we may pick $j \in V_k$, $j\neq i$. Then it is easy to check that the monomial $m=\frac{\xb^\ub}{x_i}{x_j} \in \omega_R$  and $\deg_{x_i}(m)=a_i-1$, a contradiction.

When $r_k=1$, then $n\geq 4$ and by the previous claim there is at least one non-heavy component $V_{k_1}$ in $\ub$ which is different from $V_k$. We pick $j\in V_{k_1}$ and since the monomial $m=\frac{\xb^\ub}{x_i}{x_j} \in \omega_R$  and $\deg_{x_i}(m)=a_i-1$ we obtain a contradiction.
 \end{proof}

It follows at once that
\begin{equation*}
\gcd(\xb^\ub : \xb^\ub \in \omega_R)=\prod_{i=1}^d x_i,
\end{equation*}
 where the greatest common divisor is computed in the polynomial ring $S=\KK[x_1,\dots, x_d]$.

Since $\omega_R$ is generated by monomials, one gets that $\omega_R^{-1}$ is also generated by monomials in $\KK[x_1^{\pm 1}, \dots, x_d^{\pm 1}]$.
 If  $f=\xb^\ub/\xb^\vb \in \omega_R^{-1}$ with $\xb^\ub$ and $\xb^\vb$ coprime  monomials in $S$, then $\xb^\vb$ divides the greatest common divisor of the monomials in $\omega_R$. Hence, in order to determine a system of generators for the $R$-module $\omega_R^{-1}$ it is enough to scan among the (non-reduced) fractions $f=\xb^\ub/(x_1\dots x_d)$, where $\xb^\ub$ is in the set
$$
\mathcal {B}=\left\{ \xb^\ub \in S: \sum_{i=1}^d u_i   \equiv 0 \mod 2, \quad \xb^\ub \cdot \omega_R \subseteq x_1\dots x_d R\right\}.
$$

A monomial $\xb^\ub$ is in $\mathcal{B}$ if and only if $\sum_{i=1}^d u_i \equiv 0 \mod 2$ and 
$$
x_1^{u_1+v_1-1}\cdots x_d^{u_d+v_d-1} \in R
$$ 
for all $x_1^{v_1}\cdots x_d^{v_d}$ in $\omega_R$. That is equivalent, via \eqref{eq:r}, \eqref{zero}, \eqref{two-r},  to the fact that
\begin{eqnarray}
\sum_{i=1}^d u_i \equiv d \mod 2, \text{ and }\\
\label{eq:first}\sum_{i=1}^d u_i +\sum_{i=1}^d v_i  \geq d-r_k + 2\sum_{j\in V_k} u_j + 2 \sum_{j\in V_k} v_j, \\
\nonumber \text{ for } k=1, \dots, d,\text{ and any } \xb^\vb \in \omega_R.
\end{eqnarray}
For $k=1,\dots, n $ we set 
$$
E_k=\min \left\{  \sum_{i=1}^d v_i- 2 \sum_{j\in V_k}v_k:  \xb^\vb \in \omega_R\right\}.
$$
Therefore, \eqref{eq:first} is equivalent to
\begin{equation}
\label{eq-with-ek}
\sum_{i=1}^d u_i \geq d-r_k-E_k+2\sum_{j\in V_k} u_j \text{ for } k=1,\dots, n.
\end{equation}

Before computing $E_k$ we make a simple observation regarding $d$ and  the $r_i$'s.

\begin{Claim}
$2 r_i+2 \leq d$ for all $i=1,\dots, n-1$.
\end{Claim}

\begin{proof}
Indeed, if that were not the case, then $2r_n+2 \geq 2 r_{n-1} +2>d$, hence $2 r_n\geq 2 r_{n-1} \geq d-1$. This implies $r_n+r_{n-1} \geq d-1$, equivalently that $1=\sum_{i=1}^{n-2}r_i$, which is not possible in our setup. 
\end{proof}

Next we show that $E_k$ does not depend on $k$.

\begin{Claim}{\em $E_k=2$ for $k=1,\dots, n$.}
\end{Claim}

\begin{proof}
We fix $1\leq k \leq n$.
Clearly, $E_k \geq 2$, by \eqref{two}. 
Then $E_k=2$ once we find 
\begin{equation}
\label{good-v}
\xb^\vb \in \omega_R \text{ such that } \sum_{i=1}^d v_i=2 + 2\sum_{j\in V_k} v_j.
\end{equation}
Using Eqs. \eqref{zero}, \eqref{one}, \eqref{two}, and translating $v_i=r_i+s_i$ for $i=1,\dots, n$, we observe that finding $\vb$ as in \eqref{good-v} is equivalent to finding  integers $s_1,\dots, s_n$ such that
\begin{eqnarray}
\label{s-zero} s_1,\dots, s_n \geq 0,\\
\label{s-one} \sum_{i=1}^n s_i \geq 2 s_\ell +2 r_\ell +2-d, \text{ for } 1\leq \ell \leq n, \ell\neq k, \text{ and} \\
\label{s-two} \sum_{i=1}^n s_i=2 s_k+2+2r_k-d.
\end{eqnarray} 
The $s_\ell$ represents the sum of the components of $\vb$ from $V_\ell$, each decreased by one. 
Note that \eqref{s-two} already implies that $\sum_{i=1}^n s_i \equiv d \mod 2$. 

We have two cases to consider.

{\em Case $k=n$}:

We let $s_\ell= \lfloor d/2 \rfloor -r_\ell-1$ for $\ell=1,\dots, n-1$. 
For \eqref{s-two} to hold, we must let 
\begin{eqnarray*}
s_n &=& \sum_{i=1}^{n-1} s_i-2-2r_n+d=(n-1)\lfloor d/2 \rfloor -d+r_n-(n-1)-2-2r_n+d \\
      &=& (n-1)(\lfloor d/2\rfloor -1)-r_n-1 \geq 2(\lfloor d/2 \rfloor -1)-r_n-1 \geq d-r_n-2 \geq 0.
\end{eqnarray*}
For $\ell<n$, one has $s_\ell \geq 0$ by the previous Claim. Also,   $2s_\ell+2+2 r_\ell-d$ is either $0$ or $1$, depending on $d$ being even or odd.
 Therefore, \eqref{s-one} and \eqref{s-zero} are all verified.

{\em Case $1\leq k \leq n-1$}:

We let $s_n=0$ 
 and $s_\ell= \lfloor d/2 \rfloor -r_\ell-1$ for $\ell=1,\dots, n-1$ where $\ell\neq k$. 
For \eqref{s-two} to hold, we must let 
\begin{eqnarray}
\label{sk}
s_k &=& (\sum_{1\leq i \leq n-1, i\neq k}s_i )+s_n-2-2r_k+d. 
\end{eqnarray}
Clearly, $s_k\geq 0$ since 
$d\geq 2r_k+2$.
Arguing as in the other case, for $k\neq \ell<n$ one has $s_\ell \geq 0$ and \eqref{s-one} holds.
We are left to verify that 
\begin{equation}
\label{eq-last}
\sum_{i=1}^n s_i \geq 2 s_n+2 r_n+2-d.
\end{equation}
Substituting \eqref{s-two} into the left hand side term above,  \eqref{eq-last} is equivalent to 
\begin{equation*}
s_k+r_k \geq s_n+r_n.
\end{equation*}
Using \eqref{sk} we get that
\begin{eqnarray*}
s_k+r_k &=& (\sum_{1\leq i\leq n-1, i\neq k} s_i) +s_n+d-r_k-2 \\
              &=& (\sum_{1\leq i \leq n-1, i\neq k}s_i)+s_n+r_n+(d-r_k-r_n-2)\geq s_n+r_n, 
\end{eqnarray*}
 where for the latter inequality we used the observation that $d\geq r_k+r_n+2$ in our setup.
Consequently, $s_1,\dots, s_n$ fulfil \eqref{s-zero}, \eqref{s-one}, \eqref{s-two}, and $E_k=2$.
\end{proof}

We can now finish the proof  of  Theorem \ref{thm:ng-graphs}.

Let $m=x_1^{a_1}\dots x_d^{a_d}$ be a monomial generator for $\omega_R$. Then $\deg m=\sum_{i=1}^d a_i \geq 2+2\sum_{j\in V_k}a_j$ for all $k=1,\dots, n$. In particular, $\deg m \geq 2 r_n +2$.

Let $f=\xb^\ub/(x_1\cdots x_d)$ be a  monomial in $\omega_R^{-1}$, with $\xb^\ub \in \mathcal{B}$. By \eqref{eq-with-ek},
 $$
\deg \xb^\ub=\sum_{i=1}^d u_i \geq d-r_k-2+2\sum_{j\in V_k} u_j \text{ for all } k=1,\dots, n.
$$
 Since $d>r_n+2$ in our setup, we  find a component $k_1$ such that $\sum_{j\in V_{k_1}} u_j >0$.
 
The product $m\cdot f$ is a monomial in $R$ of degree at least 
$$
(2r_n+2)+(d-r_{k_1} -2 +2\sum_{j\in V_{k_1}}u_j)-d\geq 2r_n-r_{k_1}+2 \geq 3.
$$
Consequently, $\tr(\omega_R)=\omega_R\cdot \omega_R^{-1} \subsetneq \mm_R$, a contradiction with \eqref{eq:tracemax}.
\end{proof}
 
\section{Stable set rings}

In this section we consider an algebra generated by the monomials coming from the stable sets of a graph.

Let $G$ be a finite simple graph on $[n]$ and $E(G)$ is the set of edges of $G$.    A subset $C \subset [n]$ is a {\em clique} of $G$ if $\{i, j\} \in E(G)$ for all $i, j \in C$ with $i \neq j$.  A subset $W \subset [n]$ is {\em stable} in $G$ if $\{i, j\} \not\in E(G)$ for all $i, j \in W$ with $i \neq j$.  In particular, the empty set as well as each $\{ i \} \subset [n]$ is both a clique of $G$ and a stable subset of $G$.  Let $\Delta(G)$ denote the {\em clique complex} of $G$ which is the simplicial complex on $[n]$ consisting of all cliques of $G$.  Let $\delta$ denote the maximal cardinality of cliques of $G$.  Thus $\dim \Delta(G) = \delta - 1$.  We say that $G$ is {\em pure} if $\Delta(G)$ is a pure simplicial complex, i.e. the cardinality of each maximal clique of $G$ is $\delta$.  
The {\em chromatic number} of a graph is the smallest number of colors that can be used for its vertices such that no adjacent vertices have the same color.
The graph $G$ is called {\em  perfect} if for all induced subgraphs $H$ of $G$, including $G$ itself, the chromatic number is equal to the maximal cardinality of cliques contained in $H$, see \cite[p.~165]{BB}.

Let $\KK[x_1, \ldots, x_n, t]$ denote the polynomial ring in $n + 1$ variables over the field $\KK$.  If, in general, $W \subset [n]$, then $x^W t$ stands for the squarefree monomial 
\[
x^W t = \big(\prod_{i \in W}x_i \big) \cdot t \in \KK[x_1, \ldots, x_n, t].
\]  

Let ${\rm Stab}_\KK(G)$ denote the subalgebra of $\KK[x_1, \ldots, x_n]$ which is generated by those $x^W t$ for which $W$ is a stable set of $G$.  Letting $\deg(x^W t)=1$ for any stable set $W$,  the algebra ${\rm Stab}_\KK(G)$ becomes standard graded.  We call ${\rm Stab}_\KK(G)$ the {\em stable set ring} of $G$.

 It is known \cite[Example 1.3 (c)]{OH} that ${\rm Stab}_\KK(G)$ is normal if $G$ is perfect.  It follows that, when $G$ is perfect, ${\rm Stab}_\KK(G)$ is spanned over $\KK$ by those monomials $(\prod_{i=1}^{n} x_i^{a_i}) t^q$ with $\sum_{i \in C} a_i \leq q$ for each maximal clique $C$ of $G$.  Furthermore, the canonical module $\omega_{{\rm Stab}_\KK(G)}$ of ${\rm Stab}_\KK(G)$ is spanned over $\KK$ by those monomials $(\prod_{i=1}^{n} x_i^{a_i}) t^q$ with each $a_i > 0$ and with $\sum_{i \in C} a_i < q$ for each maximal clique $C$ of $G$.  Thus \cite[Theorem 2.1 (b)]{OHjct} implies that ${\rm Stab}_\KK(G)$ is Gorenstein if and only if $G$ is pure.  

The following lemma captures a sufficient combinatorial condition for ${\rm Stab}_\KK(G)$ to be nearly Gorenstein. 

\begin{Lemma}
\label{perfect}
Let $G$ be a finite simple  perfect graph such that ${\rm Stab}_\KK(G)$ is nearly Gorenstein.
Then every connected component of $G$ is pure.
\end{Lemma}

\begin{proof}
Assume $V(G)=[n]$. Denote $R={\rm Stab}_\KK(G)$. 
Since each $x_i t$ as well as $t$ belongs to $R$, the quotient field of $R$ is the rational function field $\KK(x_1, \ldots, x_n, t)$ over $\KK$.  

Suppose $G_1$ is a connected component of $G$ which is not pure.
Let $\delta$ and $\delta_1$  denote the maximal cardinality of cliques of $G$ and of $G_1$, respectively.
Then there is an edge $\{i_0, j_0\} \in E(G_1)$ for which $i_0$ belongs to a clique $C$ of $G$ with $|C| = \delta_1$ and for which $j_0$ belongs to {\em no} clique $C$ of $G$ with $|C| = \delta_1$.

Let $z=\prod_{i=1}^{n} x_i^{a'_i} t^{q'} \in  \omega_R^{-1}$. Set $v_1=x_1\cdots x_n t^{\delta+1}$.
It is easy to check that $v_1\in \omega_R$ and that each monomial belonging to $\omega_R$ is divisible (in $\KK[x_1,\dots, x_n, t]$) by $v_1$. Hence  $a_i \geq -1$ for all $i$. Clearly, $x_{j_0}v_1\in \omega_R$ and $1\neq x_{j_0}v_1z\in R$, hence $q'\geq -\delta$.

Since $G$ is not pure, $R$ is not a Gorenstein ring and thus
$$
\tr(\omega_R)=\omega_R\cdot \omega_R^{-1}=\mm_R.
$$

Let $w' = \prod_{i=1}^{n} x_i^{a'_i} t^{q'} \in \omega_R^{-1}$ and $w = \prod_{i=1}^{n} x_i^{a_i} t^{q} \in \omega_R$ with $w'w = x_{i_0}t$.  Since $q' \geq - \delta$ and $q \geq \delta + 1$, one has $q' = - \delta$ and $q = \delta + 1$.

Let $v = x_1x_2\cdots x_n t^{\delta+1}\cdot x_{i_0}^{\delta-\delta_1}$.  One has $v\in \omega_R$ and $x_{j_0}v\in \omega_R$. 
We claim that  $w'\cdot x_{j_0}v \in \mm_R$ is divisible by $x_{i_0}x_{j_0}t$, but it is not divisible by $t^2$. 
This is clear when $\delta>\delta_1$. In case $\delta=\delta_1$, since $i_0$ belongs to a clique $C$ of $G$ with $|C| = \delta$, one has $a_{i_0} = 1$.  Thus $a'_{i_0} = 0$ and the claim is verified. 

 Thus $w'\cdot x_{j_0}v$ must be of the form $x^W t$, where $W$ is a stable set of $G$, which contradicts $\{i_0, j_0\} \in E(G)$.  Hence $\mm_R \subsetneq \tr(\omega_R)$, as desired.   
\end{proof}

Recall that the $a$-invariant of any graded algebra $R$ with canonical module $\omega_R$ is defined as 
$a(R)= -\min\{i: (\omega_R)_i\neq 0\}$. 
 
\begin{Corollary}
\label{a-inv}
If $G$ is a perfect graph   then $a({\rm Stab}_\KK(G))=-\dim \Delta(G)-2$.
\end{Corollary}

\begin{proof}
Let $\delta$ be the maximal size of a clique in $G$.
From the proof of the Lemma~\ref{perfect}, $v=x_1\cdots x_nt^{\delta+1}$ is in $(\omega_{{\rm Stab}_\KK(G)})_{\delta+1}$ and $v$ divides every monomial in    $\omega_{{\rm Stab}_\KK(G)}$. This gives the conclusion. 
\end{proof}

\begin{Theorem}
\label{ngstable}
Let $G$ be a finite simple graph with $G_1, \ldots, G_s$ its connected components and suppose that $G$ is perfect.  Let $\delta_i$ denote the maximal cardinality of cliques of $G_i$.  Then ${\rm Stab}_\KK(G)$ is nearly Gorenstein if and only if each $G_i$ is pure and $|\delta_i - \delta_j| \leq 1$ for $1 \leq i < j \leq s$.
\end{Theorem}

\begin{proof}
Suppose that ${\rm Stab}_\KK(G)$ is nearly Gorenstein.  It follows   from Lemma \ref{perfect} that each $G_i$ is pure and each ${\rm Stab}_\KK(G_i)$ is Gorenstein.  Since ${\rm Stab}_\KK(G)$ is the Segre product of ${\rm Stab}_\KK(G_1), \ldots, {\rm Stab}_\KK(G_s)$, it follows from \cite[Corollary 4.16]{HHS}  and \cite[Corollary 2.8]{HMP} that
$$
|a({\rm Stab}_\KK(G_i))-a({\rm Stab}_\KK(G_j))|\leq 1 \text{ for all }i,j.
$$ 
Corollary~\ref{a-inv} yields $|\delta_i - \delta_j| \leq 1$ for $1 \leq i < j \leq s$.  Furthermore, the ``If'' part also follows from \cite[Corollary 4.16]{HHS} and \cite[Corollary 2.8]{HMP}.
\end{proof} 

\begin{Corollary}
Let $G$ be a finite simple graph which is pefect and connected. Then the ring ${\rm Stab}_\KK(G)$ is nearly Gorenstein if and only if it is Gorenstein.
\end{Corollary}
\medskip

{\bf Acknowledgement.}   The first author was partially supported by JSPS KAKENHI 19H00637.

{}

\medskip


\begin{thebibliography}{}
 
\bibitem{BB} B.~Bollob\'as, \textit{Modern Graph Theory},  GTM 184, Springer, 1998.

\bibitem{BH} W.~Bruns, J.~Herzog, \textit{Cohen--Macaulay Rings}, Revised Ed., Cambridge Stud. Adv. Math., vol. {\bf 39}, Cambridge University Press, Cambridge, 1998.

\bibitem{BVV} W.~Bruns, W.~Vasconcelos, R.~Villarreal, \textit{Degree bounds in monomial subrings}, Illinois Journal of Mathematics {\bf 41} (1997), no. 3, 341--353.

\bibitem{Danilov} V.I.~Danilov, \textit{The geometry of toric varieties}, Russian Math. Surveys, {\bf 33} (1978), 97--154.

\bibitem{DeNegri-Hibi} E.~De~Negri, T.~Hibi, \textit{Gorenstein Algebras of Veronese Type}, J. Algebra {\bf 193} (1997), 629--639.


\bibitem{GTT} S.~Goto, R.~Takahashi, N.~Taniguchi,   \textit{Almost Gorenstein rings--towards a theory of higher dimension},  J. Pure Appl. Algebra {\bf 219} (2015), 2666--2712.

\bibitem{HHS} J.~Herzog, T.~Hibi, D.I.~Stamate, \textit{The trace of the canonical module},  Israel J. Math. {\bf 233} (2019), 133--165. 

\bibitem{HMP} J.~Herzog, F.~Mohammadi, J.~Page, \textit{Measuring the non-Gorenstein locus of Hibi rings and normal affine semigroup rings}, J. Algebra {\bf 540} (2019), 78--99. 

\bibitem{TH87}
T.~Hibi, \textit{Distributive lattices, affine semigroup rings and algebras with
straightening laws},  in: ``Commutative Algebra and Combinatorics''
(M. Nagata and H. Matsumura, Eds.), Advanced Studies in Pure Math.,
Volume 11, North--Holland, Amsterdam, 1987, pp. 93--109.

\bibitem{HM-2020} A.~Higashitani, K.~Matsushita, \textit{Conic divisorial ideals and non-commutative crepant resolutions of edge rings of complete multipartite graphs},  arXiv:2011.07714 [math.AC].

\bibitem{HM-2021} A.~Higashitani, K.~Matsushita, \textit{Levelness versus almost gorensteinness of edge rings of complete multipartite graphs}, arXiv:2102.02349 [math.AC]

\bibitem{Hochster} M.~ Hochster, \textit{Rings of invariants of tori, Cohen-Macaulay rings generated by monomials,
and polytopes}, Ann. of Math. {\bf 96} (1972), 318--337.

\bibitem{OH-jalg} H.~Ohsugi, T.~Hibi, \textit{Normal polytopes arising from finite graphs},  J. Algebra {\bf 207} (1998), 409--426.

\bibitem{OH-Ill}  H.~Ohsugi, T.~Hibi, \textit{Compressed polytopes, initial ideals and complete multipartite graphs}, Illinois J. Math {\bf 44} (2000), 391--406.

\bibitem{OH} 
H.~Ohsugi, T.~Hibi, \textit{Convex polytopes all of whose reverse lexicographic initial ideals are squarefree},   Proc. Amer. Math. Soc. {\bf 129} (2001), 2541--2546.

\bibitem{OHjct}
H.~Ohsugi, T.~Hibi, \textit{Special simplices and Gorenstein toric rings}, J. Combin. Theory, Ser. A {\bf 113} (2006), 718--725.

\bibitem{Stanley} R.~Stanley, \textit{Hilbert functions of graded algebras}, Adv. Math. {\bf 28} (1978), 57--83.

\bibitem{Villarreal} R.~H.~Villarreal, \textit{On the equations of the edge cone of a graph and some applications}, manuscripta math. {\bf 97} (1998), 309--317.


\end{thebibliography}
\end{document}